\documentclass[11pt]{article}
\usepackage[T1]{fontenc}
\usepackage[utf8]{inputenc}
\usepackage{lmodern}
\usepackage[margin=1in]{geometry}
\usepackage{amsmath,amssymb,amsthm,mathtools}
\usepackage{microtype}
\usepackage{xcolor}
\usepackage[colorlinks=true,linkcolor=blue!45!black,citecolor=blue!45!black,
  urlcolor=blue!45!black]{hyperref}
\usepackage[nameinlink,noabbrev,capitalise]{cleveref}
\allowdisplaybreaks[2]
\numberwithin{equation}{section}
\newtheorem{theorem}{Theorem}[section]
\newtheorem{lemma}[theorem]{Lemma}

\newtheorem{corollary}[theorem]{Corollary}
\theoremstyle{remark}
\newtheorem{remark}[theorem]{Remark}
\newcommand{\E}[1]{\mathbb{E}_{#1}}

\newcommand{\calP}{\mathcal P}
\hypersetup{
  pdftitle={Sharper Zarankiewicz and Diagonal Bipartite Ramsey Bounds},
  pdfauthor={Dhruv Mubayi},
  pdfsubject={Dense Zarankiewicz and diagonal bipartite Ramsey bounds}
}

\title{Sharper Zarankiewicz and Diagonal Bipartite Ramsey Bounds}
\author{Dhruv Mubayi\thanks{%
		Department of Mathematics, Statistics and Computer Science,
		University of Illinois Chicago, Chicago, IL 60607.
		Email: \texttt{mubayi@uic.edu}.
		Research partially supported by NSF Awards
		DMS-2552740 and DMS-2153576.}}

\date{September 26, 2026}

\begin{document}
\maketitle

\begin{abstract}
	We prove that there is an absolute positive constant $c$ such that every bipartite graph with $N$ vertices
	in each part and at least $N^2/2$ edges contains a complete bipartite graph $K_{t,t}$ whenever
	$N\ge c\, 2^{t}$. This  improves the classical K\H{o}v\'ari-S\'os-Tur\'an bound requiring $N$ of order $t\,2^t $. As a consequence,  the diagonal bipartite
	Ramsey number has upper bound $b(t,t)=O(2^t)$, improving the previous best  bound $b(t,t) = O(2^t\, \log t )$ due to 
	Conlon. The proof was found by GPT-6 Astra, and the method will probably have further applications.
	
	\end{abstract}

	\section{AI Declaration Statement}
The proof of the main result was found by GPT-6 Astra after several guided prompts by the author. Astra first produced a proof that was only 8 pages long and was incomprehensible to the author. As the author tried to understand the proof, many iterations of the write-up were produced, culminating in the current one in which  much of the text and explanations are due to the author. The author takes full responsibility for all the content in the paper.

\section{Introduction}\label{sec:introduction}

For positive integers $m,n,s,t$, let $z(m,n;s,t)$ denote the maximum
number of edges in a bipartite graph with specified classes of sizes
$m,n$ that contains no complete bipartite graph with $s$ vertices in
the first class and $t$ in the second. When $s=t$ we refer to this as the diagonal case. The theorem of
K\H{o}v\'ari, S\'os, and Tur\'an~\cite{KST} gives, in the 
diagonal case,
\begin{equation}\label{eq:KST}
 z(N,N;t,t)\le (t-1)^{1/t}N^{2-1/t}+(t-1)N
 \qquad (t\ge2).
\end{equation}
In a bipartite graph with density at least  $1/2$, the  counting argument showing (\ref{eq:KST}) guarantees a
$K_{t,t}$ when $N$ is of order $t2^t$. Our result removes
the factor of order $t$ in this dense regime. 

\begin{theorem}\label{thm:half}
Let $t,N$ be positive integers. If $G\subseteq K_{N,N}$,
$e(G)\ge N^2/2$, and $N\ge2^{t+81}$, then $G$ contains $K_{t,t}$.
\end{theorem}
Equivalently, Theorem~\ref{thm:half} asserts that
$z(N,N;t,t)<N^2/2$ for $N\ge c 2^{t}$ where $c=2^{81}$, whereas substituting $N=c2^t$ into (\ref{eq:KST}) gives, for fixed $c>0$ as $t\to\infty$,
\[
z(N,N;t,t)\le
\left(
\frac12+\frac{\log t-\log c}{2t}
+O_c\!\left(\frac{(\log t)^2}{t^2}\right)
\right)N^2.
\]This shows explicitly why the K\H{o}v\'ari--S\'os--Tur\'an
estimate alone does not establish the desired conclusion
when $N=c2^t$ with $c$ fixed.  The exact bound (\ref{eq:KST}) requires $c$ to grow on the order of $t$ to be effective in the range $e(G) > N^2/2$. The constant $2^{81}$
has not been optimized. More generally, \Cref{thm:density} shows that a bipartite graph
of density $0<\rho<1$ with $N$ vertices in each class contains
$K_{t,t}$ whenever
\[
N\ge C_\rho\,\rho^{-t},
\qquad C_\rho=2^{41}[\rho(1-\rho)]^{-20}.
\]
Thus, for each fixed density, the sufficient number of vertices
is a constant multiple of $\rho^{-t}$; taking $\rho=1/2$
recovers the bound $N\ge2^{t+81}$.

For $1\le s\le t$, let $b(s,t)$ be the least $N$ such that every
red--blue coloring of $K_{N,N}$ contains a monochromatic $K_{s,t}$,
where the part of size $s$ (and $t$) can lie in either part of the $K_{N, N}$. The bipartite Ramsey
problem was introduced by Beineke and Schwenk~\cite{BeinekeSchwenk};
its connection with Zarankiewicz numbers was developed by
Irving~\cite{Irving}. Thomason~\cite{Thomason} proved
\begin{equation}\label{eq:Thomason}
 b(s,t)\le2^s(t-1)+1.
\end{equation}
In particular, $b(t,t)\le(t-1)2^t+1$. Conlon~\cite{Conlon}
subsequently improved the bound in the diagonal case to
\begin{equation}\label{eq:Conlon}
 b(t,t)\le(2+o(1))2^t\log_2t.
\end{equation}
For lower bounds, independently coloring each edge red or blue with
equal probability gives an expected number of monochromatic copies
of $K_{t,t}$ equal to $2\binom Nt^2 2^{-t^2}$. The first-moment
argument and Stirling's formula therefore give
 $b(t,t)\ge(1/e+o(1))t\,2^{t/2}$.
Hattingh and Henning~\cite{HattinghHenning} used the Lov\'asz Local
Lemma to obtain the stronger bound
 $b(t,t)\ge(\sqrt 2/e+o(1))t\,2^{t/2}$,
which remains the best known asymptotic lower bound; see also
Conlon~\cite{Conlon}. Thus the known lower and upper bounds remain
exponentially apart.
By choosing a color with at least half the edges in a 2-coloring of $E(K_{N,N})$ and applying
\Cref{thm:half} we immediately obtain the following improvement over \eqref{eq:Conlon}.

\begin{corollary}\label{thm:ramsey}
For every positive integer $t$, we have $b(t,t)\le2^{t+81}$.
\end{corollary}

\Cref{thm:half} does not improve  the exponent $2-1/t$ in the classical
Zarankiewicz problem with $t$ fixed and $N$ tending to infinity.
For a broader account of bipartite extremal problems, see
F\"uredi and Simonovits~\cite{FurediSimonovits}.

\Cref{sec:ideas} explains the ideas of the proof and provides extensive motivation for introducing a definition (called the potential function) that is the crucial new ingredient in the proof; we expect that this idea may have further applications. 
The main technical statement is the one-step density theorem,
\Cref{thm:one-step}, stated in \Cref{sec:one-step-statement}.
Sections~\ref{sec:log}--\ref{sec:stability} establish the estimates
needed for its proof, which is given in \Cref{sec:induction}.
We then deduce a general density criterion in \Cref{sec:application}
and prove \Cref{thm:half,thm:ramsey} in \Cref{sec:half-application}.

\section{Proof ideas}\label{sec:ideas}

The proof builds a copy of $K_{t,t}$ by choosing edges one by one,
unless a preliminary density criterion already supplies the remaining
biclique. At each selection step, we restrict our attention to the
bipartite graph induced by the neighborhoods of the chosen endpoints,
after deleting those endpoints, since all future choices must lie in
this residual graph. The main idea is to define a potential function
that measures how long we can continue this process: the potential
starts as a large number and decreases by only a small amount at each
selection step.
We first explain this algorithm with more precise notation  and the requirements on the
potential, then motivate its particular choice (which may not be unique).

\subsection{Choosing vertices on both sides}

The most straightforward approach is to choose all $t$ vertices on one
side and look for $t$ common neighbors on the other. If
$L,R$ are the classes of a bipartite graph with $|L|=|R|=N\ge t$ and
$N_G(S)=\bigcap_{x\in S}N_G(x)$, then double counting gives
\begin{equation}\label{eq:one-shot}
 \frac1{\binom Nt}\sum_{\substack{S\subseteq L\\|S|=t}}|N_G(S)|
 =\frac1{\binom Nt}\sum_{y\in R}\binom{d(y)}t.
\end{equation}
For example, if every vertex in $R$ has degree $N/2$, this average
is
\[
 N\frac{\binom{N/2}t}{\binom Nt}
 =(1+o(1))N2^{-t}
 \qquad\text{when }t^2/N\longrightarrow0.
\]
When $N=c2^t$ with $c$ fixed, the average is approximately $c$.
Hence to be guaranteed $t$ common neighbors we would need a factor of order $t$ in the classical
counting argument.

Instead, select one vertex on each side at a time. Suppose that
$i$ vertices have already been chosen in each class and form a
$K_{i,i}$. Keep unused left and right candidate sets, each vertex
of which is adjacent to every previously selected vertex on the
opposite side. Let $J$ be the graph between these candidate sets.
For an edge $xy\in E(J)$, select $x,y$ and replace the candidate
sets by
\begin{equation}\label{eq:construction-step}
 L'=N_J(y)\setminus\{x\},\qquad
 R'=N_J(x)\setminus\{y\}.
\end{equation}
Write $J'_{xy}=J[L',R']$ for the resulting graph. Every
$K_{k-1,k-1}$ in $J'_{xy}$ extends to a $K_{k,k}$ in $J$ by
adding $x,y$. Thus it suffices either to find the required biclique
directly in the current graph or to choose an edge whose residual
graph contains the smaller biclique. The problem in the second case
is to leave enough useful candidates to continue.

An ideal calculation explains the exponential growth $2^t$ as follows. Suppose both
candidate sets have size $M_i$, their density remains $1/2$, and
the chosen vertices have half the opposite candidate set as
neighbors. Then
\[
 M_{i+1}=M_i/2-1,
 \qquad
 M_i=\frac{N+2}{2^i}-2.
\]
If $N=c2^t$, then $M_t=c-2+2^{1-t}$, so any constant $c>2$
leaves a nonempty candidate set after $t$ ideal steps. In an arbitrary graph,
the candidate sizes and their density vary together. We therefore
attach a nonnegative real number to the current graph and to the
number $k$ of vertices still required on each side. This number
measures the remaining supply of candidates after accounting for
future neighborhood restrictions. We call it a \emph{potential
function}. The proof chooses an edge for which this potential
falls by only a controlled amount; this one-step inequality
replaces the  recurrence for $M_i$.

\subsection{Requirements for a potential function}

Our next goal is to identify the properties that the potential
must satisfy. Before choosing a formula, consider nonnegative
potential functions $\mathcal P_k(J)$, where $J$ is the current
bipartite graph and $k$ is the number of vertices still needed on
each side. We want thresholds $F_k$ such that
\begin{equation}\label{eq:abstract-bound}
 \mathcal P_k(J)>F_k
 \quad\Longrightarrow\quad J\text{ contains }K_{k,k}.
\end{equation}
We call $F_k$ an \emph{inductive potential threshold}.
There are four requirements. First, the potential must be defined
for every graph produced by the construction, including graphs with
unequal or empty classes and all densities. Second, positive potential
at $k=1$ must guarantee an edge. We arrange this by assigning potential
zero to every edgeless graph, including graphs with an empty class.

Third, for $k\ge2$, a graph with potential above its threshold must
allow the algorithm to proceed. We first apply a preliminary
density criterion, proved in \Cref{lem:coarse}, to the current graph
and to its residual graphs. If one of these tests succeeds, it gives
the required biclique directly, using the extension property above.
If none succeeds, we require the one-step implication
\begin{equation}\label{eq:abstract-survival}
 \mathcal P_k(J)>F_k
 \quad\Longrightarrow\quad
 \mathcal P_{k-1}(J'_{xy})>F_{k-1}
 \quad\text{for some }xy\in E(J).
\end{equation}
The numerical hypotheses under which we prove this implication
are stated precisely in \Cref{thm:one-step}.

Finally, let $G$ be a starting graph with $N$ vertices in each class
and density exactly $1/2$; the reduction from density at least $1/2$
is justified in \Cref{sec:half-application}. We require
\[
\mathcal P_t(G)\ge a_0N2^{-t}
\qquad\text{and}\qquad
F_k\le F_*\quad\text{for every }k,
\]
where $a_0>0$ and $F_*$ are absolute constants. If
$N>(F_*/a_0)2^t$, then
\[
\mathcal P_t(G)\ge a_0N2^{-t}>F_*\ge F_t.
\]
At each stage, either the preliminary criterion finishes the
construction or \eqref{eq:abstract-survival} reduces the number of
vertices still needed by two. If the process reaches $k=1$, the
positive potential guarantees an edge. Adding back the previously
chosen endpoints gives a $K_{t,t}$. This proves
\eqref{eq:abstract-bound} by induction once the one-step implication
has been established in the case where the preliminary tests fail.

The uniform bound $F_k\le F_*$ makes the sufficient starting size
a constant multiple of $2^t$. If instead the thresholds grew
proportionally to $k$, the same argument would require a starting
size of order $t2^t$.

We choose nonnegative losses $\ell_k$ with
$\sum_{k\ge2}\ell_k<\infty$, and uniformly bounded thresholds
$F_k$ satisfying $F_k-F_{k-1}\ge\ell_k$.
In the case requiring the one-step estimate, it suffices to prove
the stronger expectation bound
\begin{equation}\label{eq:abstract-loss}
 \E{\mathrm{edge}}\mathcal P_{k-1}(J'_{xy})
 \ge\mathcal P_k(J)-\ell_k
 >F_k-\ell_k
 \ge F_{k-1}.
\end{equation}
Since the average residual potential exceeds $F_{k-1}$, some edge
satisfies \eqref{eq:abstract-survival}. Here
\[
 \E{\mathrm{edge}}W_{xy}
 =\frac1{e(J)}\sum_{xy\in E(J)}W_{xy}
\]
is expectation over a uniformly chosen edge of $J$.
A fixed loss of $2$ per step would instead give
$F_k=F_1+2(k-1)$ and lead only to a bound of order $t2^t$.

\subsection{The form of the potential function}

All logarithms here and throughout the paper are natural unless otherwise specified. Our goal in this section is to explain how we came to  our potential function -- which is  chosen to be a product of a size factor
and a density factor --  and then to derive its dependence on $k$.

Let the current sizes of the vertex parts of the allowable vertices be $m,n$, let the current density be $\rho$,
and suppose that $k$ vertices remain to be chosen on each side.
A potential need not have the form $f(m,n)g_k(\rho)$: it could
depend jointly on $m,n,k,\rho$, or on further features of the graph.
However, the product form appears to be a natural simple choice that works for our proof.

 The function $f(m,n)$
measures the effective supply of vertices, while $g_k(\rho)$
accounts for the density of the graph.
Separating these roles lets us use one size function at every
stage and change the density factor as $k$ decreases. It also
has the following useful property: for positive factors,
\[
 \log\bigl(f(m,n)g_k(\rho)\bigr)
 =\log f(m,n)+\log g_k(\rho).
\]
We can therefore combine a logarithmic size estimate with a
logarithmic density estimate, even though the size and density
of a link depend on the same chosen edge and are not independent. 

We want $f$ to measure in some form the number of available vertices, so we impose
\[
 f(\lambda m,\lambda n)=\lambda f(m,n),\qquad f(N,N)=N.
\]
The first condition says that scaling both classes scales their
effective size by the same factor. Symmetry in $m,n$ and
comparability with $\min\{m,n\}$ are also natural requirements,
since the smaller class limits a balanced biclique.
Ignore the deletion of the selected vertices for the ensuing discussion.
In an ideal step of density $\rho$, both class sizes are multiplied
by $\rho$, the density remains $\rho$, and $k$ decreases to $k-1$.
Before the step, the potential is $f(m,n)g_k(\rho)$.
The step multiplies the size factor by $\rho$, but also
reduces the target from $K_{k,k}$ to $K_{k-1,k-1}$.
To balance these changes and preserve a large potential
for the next stage, we choose
$g_k(\rho)=\rho g_{k-1}(\rho)$.
The potential is then unchanged in this ideal model.
Iterating the recurrence gives
\[
g_k(\rho)=\rho^k g_0(\rho).
\]
This motivates the factor $\rho^k$; the actual proof will
allow small, controlled losses at each step.
 We therefore obtain
\begin{equation}\label{eq:potential-family}
 \mathcal P_k(J)=f(m,n)\rho^k w(\rho),
\end{equation}
where $w$ is independent of $k$. The exponent $k$ records the
$k$ neighborhood restrictions still to come. At starting density
$1/2$ and with equal classes, the potential is
$$\mathcal P_t(G)=N2^{-t} w(1/2).$$ 
The correction $w$ remains to be chosen so
that the actual losses, including deletion and changes in density,
are summable. 

\subsection{Deriving the harmonic mean}

The aim in this section is  to explain the choice of $f$ that allows us to compute various averages over  edges. The symmetry, scaling, and smaller-class
requirements above leave many choices, though there is a particularly convenient one.

Write $e=e(J)=\rho mn>0$. For an edge $xy$, put
$u=d_J(y)$ and $v=d_J(x)$, the sizes of the full neighborhoods
before deleting $x,y$. If $m_+,n_+$ count the nonisolated vertices
in the two original classes, then
\begin{equation}\label{eq:reciprocal-motivation}
 \E{\mathrm{edge}}\left(\frac1u+\frac1v\right)
 =\frac{m_++n_+}{e(J)}.
\end{equation}
Indeed, a vertex of positive degree $d$ is an endpoint of exactly
$d$ edges, and contributes $1/d$ on each one. The degree bias
in sampling an endpoint of a uniform edge therefore cancels
when its degree is inverted.
We will need a lower bound on the average of $\log f(u,v)$.
Since concavity of the logarithm yields
\[
\frac1e\sum_{xy\in E(J)}\log f(u,v)
\ge
-\log\left(\frac1e\sum_{xy\in E(J)}\frac1{f(u,v)}\right),
\]
it suffices to bound the average of $1/f(u,v)$ from above.
The identity \eqref{eq:reciprocal-motivation} makes this
particularly simple if we choose $1/f(m,n)$ to be linear
in $1/m$ and $1/n$.
Within this reciprocal-linear family, symmetry
requires equal coefficients, and $f(N,N)=N$ fixes each coefficient
to be $1/2$. We have thus arrived at the harmonic mean:
\begin{equation}\label{eq:H}
 f(m,n)=H(m,n)=\frac{2mn}{m+n}\qquad(m,n>0).
\end{equation}
If either class is empty, set $f(m,n)=H(m,n)=0$.
It satisfies all the desired size conditions:
\[
 \min\{m,n\}\le H(m,n)\le2\min\{m,n\},\qquad
 H(N,N)=N,\qquad H(\lambda m,\lambda n)=\lambda H(m,n).
\]
In particular, unequal candidate classes can be retained without
discarding vertices merely to balance them.
Write $e=e(J)=\rho mn$. For an edge $xy$, with $x\in L$
and $y\in R$, the full link has vertex classes $N_J(y)$ and
$N_J(x)$, of sizes $d_J(y)$ and $d_J(x)$, respectively.
Since $1/H(s,t)=\frac12(1/s+1/t)$, we have
\[
\frac1e\sum_{xy\in E(J)} \frac1{H(d_J(y),d_J(x))} =\frac1{2e}\sum_{xy\in E(J)} \left(\frac1{d_J(x)}+\frac1{d_J(y)}\right) =\frac{m_++n_+}{2e} \le\frac{m+n}{2\rho mn} =\frac1{\rho H(m,n)}.
\]
Here $m_+,n_+$ count the nonisolated vertices in the two
classes. Each such vertex of degree $d$ contributes $1/d$
on each of its $d$ incident edges, giving a total contribution
of one.
Applying concavity of the logarithm to the positive numbers
$1/H(d_J(y),d_J(x))$ therefore gives
\begin{equation}\label{eq:ideas-size}
		\frac1e\sum_{xy\in E(J)}
		\log H(d_J(y),d_J(x))
		\ge
		-\log\left(
		\frac1e\sum_{xy\in E(J)}
		\frac1{H(d_J(y),d_J(x))}
		\right)
		\ge\log\bigl(\rho H(m,n)\bigr).
\end{equation}
These calculations show why the harmonic mean is perhaps the most natural candidate for $f(n,m)$.

\subsection{Logarithms and linearity of expectation}

Our goal is to combine the size and density estimates for the
 chosen edge and retain the gain needed to pay for
later errors. Given an edge $xy$ write $q=q_{xy}$ for the density of its full link.
The current density $\rho$ and the link density $q$ need not agree.
For example, $C_8$ with four vertices in each class has density
$1/2$; every full link has density $3/4$, whereas deletion of the
chosen endpoints leaves a graph of density zero.

The size $H(u,v)$ and density $q$ both depend on the chosen edge.
Separate lower bounds on their arithmetic means cannot simply
be multiplied to bound the mean of their product. Logarithms
resolve this difficulty via
\[
 \log\bigl(H(u,v)q^r\bigr)=\log H(u,v)+r\log q,
\]
as expectations of the two terms can be added without  any independence
assumption. One important question is therefore whether
$\E{\mathrm{edge}}\log q$ is at least $\log\rho$.

Define the \emph{logarithmic gain}
\begin{equation}\label{eq:ideas-delta}
 \delta=\E{\mathrm{edge}}\log(q/\rho)
       =\log\left(
          \frac{\exp(\E{\mathrm{edge}}\log q)}{\rho}\right).
\end{equation}
The numerator in the rightmost expression is the geometric mean of
the full-link densities. Thus $e^\delta$ is the factor by which
this geometric mean exceeds the original density.
Using (\ref{eq:ideas-delta}), linearity of expectation, and
\eqref{eq:ideas-size}, we obtain, for $r\ge0$,
\begin{equation}\label{eq:ideas-log}
\begin{aligned}
 \E{\mathrm{edge}}\log\bigl(H(u,v)q^r\bigr)
 &=\E{\mathrm{edge}}\log H(u,v)+r\,\E{\mathrm{edge}}\log q\\
 &\ge\log\bigl(\rho H(m,n)\bigr)+r(\log\rho+\delta)
 =\log\bigl(H(m,n)\rho^{r+1}\bigr)+r\delta.
\end{aligned}
\end{equation}

We will prove in \Cref{sec:log} that $\delta\ge0$, even though
individual links can have density below $\rho$. Thus $r\delta$
is a nonnegative gain for every $r\ge0$. Nonnegativity alone
suffices for the preliminary induction, but the refined
argument that we will need in the proof needs quantitative control of the errors.
Hence, the proof will also establish
\eqref{eq:log-gap}, which bounds explicit nonnegative sums
measuring degree and common-neighbor irregularity by $\delta$.
In \Cref{sec:stability}, we use these bounds to prove
\eqref{eq:ideas-stability}: both the signed mean of $q-\rho$
and its second moment are at most $10^4\rho^{-5}\delta$.

These moment bounds make the averaged density error at most
a calculable coefficient times $\delta$. The later argument
checks that this coefficient is at most $r$, allowing the gain
$r\delta$ to absorb the error.

\subsection{The warm-up: recovering the usual bound}

We first show that the uncorrected potential (i.e., without the term $w(\rho)$) recovers the classical bound, and identify the loss responsible for its factor of $t$.

First take $w\equiv1$, and write
\[
 B_k(J)=H(m,n)\rho^k.
\]
If either class is empty, set $B_k(J)=0$ without assigning a density.

The full link contains the selected vertices $x,y$. Before
continuing the algorithm, we must delete both vertices,
together with all edges incident to them, so that neither
vertex can be selected again. The resulting graph $J'_{xy}$
has class sizes $u-1,v-1$. Let $q'$ denote its density when
both classes are nonempty.

We show that this deletion reduces the uncorrected potential
by at most $2$. Suppose first that $u,v\ge2$ and $k\ge2$.
Within the full link, $x$ is adjacent to all $v$ vertices in
the opposite class, and $y$ is adjacent to all $u$ vertices
in its opposite part. Since the edge $xy$ belongs to both
stars, deleting the two vertices removes exactly $u+v-1$
edges. Therefore
\[
uvq=(u+v-1)+(u-1)(v-1)q',
\]
or equivalently,
\[
q=\frac{u+v-1}{uv}\cdot1
+\frac{(u-1)(v-1)}{uv}\,q'.
\]
The two coefficients are nonnegative and sum to one.
Thus $q$ is a convex combination of $1$ and $q'$. Convexity
of $z\mapsto z^{k-1}$ gives
\[
q^{k-1}
\le \frac{u+v-1}{uv}
+\frac{(u-1)(v-1)}{uv}(q')^{k-1}.
\]
Multiplying by $H(u,v)=2uv/(u+v)$, we obtain
\[
\begin{aligned}
 H(u,v)q^{k-1}
 &\le\frac{2(u+v-1)}{u+v}+\frac{2(u-1)(v-1)}{u+v}(q')^{k-1}\\
 &\le2+\frac{2(u-1)(v-1)}{u+v-2}(q')^{k-1}
 =H(u-1,v-1)(q')^{k-1}+2=B_{k-1}(J'_{xy})+2.
\end{aligned}
\]

If $u=1$ or $v=1$, the full link is a star, so $q=1$ and
$H(u,v)<2$. The residual graph has an empty class and hence
potential zero. The same inequality therefore holds in
this case as well. The deletion loss is consequently at
most $2$, independently of the class sizes, the density,
and the number of remaining steps.

Taking $r=k-1$ in \eqref{eq:ideas-log}, and using the
arithmetic--geometric mean inequality with $\delta\ge 0$, gives
$\E{\mathrm{edge}}H(u,v)q^{k-1}\ge \exp(\log B_k(J))=B_k(J)$.
Hence
\begin{equation}\label{eq:ideas-coarse-recurrence}
 \E{\mathrm{edge}}B_{k-1}(J'_{xy})\ge B_k(J)-2.
\end{equation}
Starting with $C_1=0$ and $C_k=C_{k-1}+2$, induction gives
\[
 B_k(J)>C_k=2(k-1)
 \quad\Longrightarrow\quad J\text{ contains }K_{k,k}.
\]
Indeed, $B_1(J)>0$ guarantees an edge. For $k\ge2$,
\eqref{eq:ideas-coarse-recurrence} gives an edge whose residual
graph has $B_{k-1}>2(k-2)$, and the induction hypothesis supplies
a $K_{k-1,k-1}$ that extends to $K_{k,k}$.
At density $1/2$ with equal classes, this criterion guarantees
a $K_{k,k}$ when $N2^{-k}>2(k-1)$.
The factor of order $k$ comes from accumulating a constant loss
of $2$ over $k-1$ steps.

\subsection{Making the loss summable}

We next choose the density correction so that the losses add to a
constant instead of a growing function with the number of steps.
The sharper bound that we will prove, for $r\ge1$, is
\begin{equation}\label{eq:ideas-deletion}
 H(u,v)q^r\le H(u-1,v-1)(q')^r
                  +\min\{2,5r q^{r-1}\}.
\end{equation}
For $q=1/2$ and $r=k-1$, the second error term is
$5(k-1)2^{-(k-2)}$, whose sum over $k\ge2$ is $20$.
More generally, the error is summable when $q$ stays below a
fixed constant less than one.

The obstruction to summing the deletion losses is that the
 density $q$ may approach one, so the bound
$5r q^{r-1}$ need not be small. High density helps in finding
a biclique, but the number of available vertices on each side
also matters. We therefore introduce the weight
\[
w(\rho)=[\rho(1-\rho)]^M
\]
for a fixed positive integer $M$, and use the corrected potential
$B_k(J)w(\rho)$. This weight has the fixed positive value
$w(1/2)=4^{-M}$ and tends to zero as $\rho$ approaches either
endpoint.

The preliminary criterion already gives a $K_{k,k}$ if
$B_k(J)>2(k-1)$. We may therefore concentrate on graphs satisfying
the  upper bound $B_k(J)\le2(k-1)$.
We seek a one-step estimate when the corrected potential
$B_k(J)w(\rho)$ exceeds $F_k$, where $F_k\ge A$ and $A>0$ is
independent of $k$ (we will take $A=2^{40}$). In this case,
\[
A<B_k(J)w(\rho)
\le 2(k-1)w(\rho)
\le 2k\,w(\rho),
\]
and hence
\begin{equation}\label{eq:weight-nontrivial}
 [\rho(1-\rho)]^M>\frac{A}{2k}.
\end{equation}
Although $w(\rho)$ is small, the upper bound
$2(k-1)w(\rho)$ grows with $k$, so these conditions are not
incompatible for all $k$.
Conversely, if \eqref{eq:weight-nontrivial} fails while
$B_k(J)\le2(k-1)$, then
\begin{equation}\label{eq:easyest}
 B_k(J)w(\rho)
 \le 2(k-1)\frac{A}{2k}
 =A\left(1-\frac1k\right)<A\le F_k.
\end{equation}
Thus a graph with potential above its threshold either satisfies
the preliminary criterion or lies in the range
\eqref{eq:weight-nontrivial}. For fixed $k$, this range excludes
densities sufficiently close to $0$ or $1$.
The refined one-step estimate in this range will allow the
thresholds $F_k$ to remain bounded as $k$ grows. Its complete
hypotheses also include numerical bounds on the residual graphs,
as stated in \Cref{thm:one-step}.

Controlling density changes is the main technical issue. In
addition to $\delta\ge0$, we will show
\begin{equation}\label{eq:ideas-stability}
0\le\E{\mathrm{edge}}(q-\rho)\le10^4\rho^{-5}\delta, \qquad \E{\mathrm{edge}}(q-\rho)^2\le10^4\rho^{-5}\delta.
\end{equation}
The upper bound on the average density change is especially
important. 
Estimating the average density change from its second moment by
Cauchy--Schwarz would give only a bound proportional to
$\sqrt{\delta}$. For fixed $k$, this bound exceeds the available
logarithmic gain of order $k\delta$ when $\delta$ is sufficiently
small. We therefore estimate the average directly, retaining
the signs of the individual density changes and exploiting
cancellation to obtain a bound proportional to $\delta$.

The purpose of \eqref{eq:ideas-stability} is to control how
changes in density affect the correction factor in the potential.
A Taylor estimate produces a linear term involving
$\E{\mathrm{edge}}(q-\rho)$ and a quadratic error involving
$\E{\mathrm{edge}}(q-\rho)^2$. The bounds above therefore make
the resulting averaged error proportional to $\delta$, with
a coefficient depending on the density and the remaining
number of steps. We can then compare this coefficient with
that of the available logarithmic gain $r\delta$.

We prove these bounds directly in \Cref{sec:stability}, using
sums of squares that measure deviations of degrees and common
neighborhoods from their expected values. These are connected  
with box norms and quasirandomness
(see \Cref{sec:boxnorm}). The argument requires explicit estimates linear in $\delta$,
which we establish directly in \Cref{sec:stability}.

After a Taylor estimate and a separate bound on very dense links,
the coefficient of the density error is bounded, up to a constant
depending on $M$, by
\[
 \rho^{-5}(1-\rho)^{-2}\log(2k).
\]
In the range \eqref{eq:weight-nontrivial},
\[
 \rho^{-5}(1-\rho)^{-2}
 \le[\rho(1-\rho)]^{-5}
 <\left(\frac{2k}{A}\right)^{5/M}.
\]
For $M>5$, this coefficient grows more slowly than $k$.
Thus a sufficiently large fixed correction exponent and constant
$A$ ensure that the logarithmic gain covers the density error. The explicit proof uses
$M=20$ and $A=2^{40}$; neither numerical choice is intended
to be optimal.

\subsection{The one-step density theorem}\label{sec:one-step-statement}

We now put all of the previously described ingredients together and state the theorem whose proof
occupies the following sections. For a bipartite graph $J$ with
class sizes $m,n>0$ and density $\rho$, define
\begin{equation}\label{eq:potential}
 \calP_k(J)=H(m,n)\rho^{k+20}(1-\rho)^{20}
 =B_k(J)[\rho(1-\rho)]^{20}.
\end{equation}
Set $\calP_k(J)=0$ if either class is empty. In particular,
edgeless graphs and graphs of density one have potential zero.
The three factors have separate roles: $H$ measures the available
vertices, $\rho^k$ accounts for the remaining neighborhood steps,
and $[\rho(1-\rho)]^{20}$ controls the endpoint densities.
For the thresholds, set
\begin{equation}\label{eq:Fdef}
 A=2^{40},\qquad
 F_k=A\exp\!\left(\sum_{j=2}^k\frac1{j^2}\right).
\end{equation}
Thus $F_1=A$, $F_k=F_{k-1}e^{1/k^2}$ for $k\ge2$, and
\begin{equation}\label{eq:F-upper-front}
 F_k<2A=2^{41}.
\end{equation}
For the last inequality, one may use the elementary estimate
\[
 \sum_{j=2}^{\infty}\frac1{j^2}
 \le\frac14+\frac19+\frac1{16}+\int_4^\infty x^{-2}\,dx
 =\frac{97}{144}<\log2.
\]

\begin{theorem}[One-step density theorem]\label{thm:one-step}
Let $k\ge2$ be an integer and let $G$ be a bipartite graph.
For each edge $xy\in E(G)$, let
\[
 G'_{xy}=G[N_G(y)\setminus\{x\},N_G(x)\setminus\{y\}]
\]
be its residual link. Suppose that
\begin{equation}\label{eq:parent-density-hyp}
 B_k(G)\le2(k-1)
\end{equation}
and
\begin{equation}\label{eq:residual-density-hyp}
 B_{k-1}(G'_{xy})\le2(k-2)
 \qquad\text{for every }xy\in E(G).
\end{equation}
If $\calP_k(G)>F_k$, then
\begin{equation}\label{eq:expected-reduction}
 \E{\mathrm{edge}}\calP_{k-1}(G'_{xy})
 \ge\calP_k(G)-\frac{F_{k-1}}{k^2}
 >F_{k-1}.
\end{equation}
In particular, some edge satisfies
$\calP_{k-1}(G'_{xy})>F_{k-1}$.
\end{theorem}

The hypotheses involve only the sizes and densities of $G$ and
its residual links; no forbidden-subgraph assumption is imposed.
 If
\eqref{eq:parent-density-hyp} fails, \Cref{lem:coarse} gives
a $K_{k,k}$ in $G$. If \eqref{eq:residual-density-hyp} fails,
the same criterion gives a $K_{k-1,k-1}$ in a residual link,
which extends to a $K_{k,k}$ in $G$.

Sections~\ref{sec:log}--\ref{sec:stability} prove the logarithmic,
deletion, and density-change estimates needed for
\Cref{thm:one-step}. We combine them to prove the theorem in
\Cref{sec:induction}, then give the complete induction and the
deduction of the main results in
Sections~\ref{sec:application} and~\ref{sec:half-application}.

\section{The full-link inequality: Proof of (\ref{eq:ideas-log})}
\label{sec:log}

We begin the proof of \Cref{thm:one-step} by establishing
\eqref{eq:ideas-log}, together with the quantitative bound
\eqref{eq:log-gap}. The latter controls
degree and common-neighbor errors and will be used in
\Cref{sec:stability} to bound changes in the link densities.

The logarithmic averages used here are related to the logarithmic calculus
of Li and Szegedy~\cite{LiSzegedy}, developed in the study of
Sidorenko's conjecture on lower bounds for homomorphism densities of
bipartite graphs~\cite{Sidorenko}. We give a self-contained proof, including
the quantitative bound \eqref{eq:log-gap} needed below.

Let $J$ be a bipartite graph with vertex classes $L,R$, where
\[
|L|=m,\qquad |R|=n,\qquad
e=|E(J)|>0,\qquad \rho=\frac{e}{mn}.
\]
Write $N(x)=N_J(x)$ and $d(x)=|N(x)|$.
For $X\subseteq L$ and $Y\subseteq R$, let $e_J(X,Y)$
denote the number of edges between $X$ and $Y$.
Graphs with no edges, including graphs with an empty class,
have zero potential and do not satisfy the hypothesis
$\calP_k(G)>F_k$ of \Cref{thm:one-step}.
All logarithms are natural.

For an edge $ij$, with $i\in L$ and $j\in R$, its
\emph{full link} is the bipartite graph with classes
$N(j)$ and $N(i)$, containing every edge of $J$ between
these classes. Its density is
\begin{equation}\label{eq:link-density}
	q_{ij}=\frac{e_J(N(j),N(i))}{d(j)d(i)}.
\end{equation}
Since $ij$ is an edge, $i\in N(j)$ and $j\in N(i)$.
Thus the full link still contains the selected vertices
and, in particular, $q_{ij}>0$.
Recall that
\[
H(s,t)=\frac{2st}{s+t}
\qquad(s,t>0).
\]
The two part sizes of the full link of $ij$ are $d(j)$
and $d(i)$, so their harmonic mean is $H(d(j),d(i))$.

\subsection{The logarithmic gain and the error sums}

Define
\begin{equation}\label{eq:delta}
	\delta=\frac1e\sum_{ij\in E(J)}
	\log\frac{q_{ij}}{\rho}.
\end{equation}
The geometric mean of the full-link densities is therefore
$\rho\exp(\delta)$. We will prove that $\delta\ge0$ and,
more precisely, that $\delta$ bounds the sum of three
nonnegative quantities measuring irregularity.
For $i,i'\in L$, define
\begin{equation}\label{eq:common-neighbour-sums}
	C_{ii'}=|N(i)\cap N(i')|,
	\qquad
	W_{ii'}=\sum_{j\in N(i)\cap N(i')}\frac1{d(j)}.
\end{equation}
Thus $C_{ii'}$ counts common neighbors, while $W_{ii'}$
gives a common neighbor of smaller degree a larger weight.
Every sum over $i,i'\in L$ includes all ordered pairs,
including those with $i=i'$.
The row sums of $W$ have a simple graph interpretation:
\begin{equation}\label{eq:weighted-codegree-sums}
	\sum_{i'\in L}W_{ii'}
	=\sum_{j\in N(i)}\frac{|N(j)|}{d(j)}
	=d(i).
\end{equation}
Consequently,
\[
\sum_{i,i'\in L}W_{ii'}=e,
\qquad
\sum_{i,i'\in L}\frac{d(i)d(i')}{e}=e.
\]
These identities will allow several terms to cancel in
the proof.
We use the  function
\[
\psi(s,t)=s\log(s/t)-s+t
\qquad(s\ge0,\ t>0),
\]
with $\psi(0,t)=t$ and $\psi(0,0)=0$.
Its nonnegativity and quadratic lower bound are recorded
in \Cref{lem:remainder}.
Define
\begin{equation}\label{eq:three-sums}
\begin{aligned}
 S_L&=\frac1e\sum_{i\in L}\psi\left(d(i),\frac em\right),
 \qquad S_R=\frac1e\sum_{j\in R}\psi\left(d(j),\frac en\right),\\
 S_C&=\frac1e\sum_{i,i'\in L}\psi\left(W_{ii'},\frac{d(i)d(i')}{e}\right).
\end{aligned}
\end{equation}
If $d(i)d(i')=0$, then $W_{ii'}=0$, so the corresponding
term in $S_C$ is well defined and equals zero.
Each of the three sums is nonnegative.

The quantities $S_L,S_R$ compare the degrees with their
respective average degrees $e/m,e/n$. The quantity $S_C$
compares the weighted common-neighbor sums with
$d(i)d(i')/e$. Controlling these three quantities by $\delta$
is the quantitative information needed later.

\begin{lemma}
	\label{lem:log}
	With the notation above,
	\begin{equation}\label{eq:log-gap}
		\delta\ge S_L+S_R+S_C\ge0.
	\end{equation}
	Moreover, for every real $r\ge0$,
	\begin{equation}\label{eq:edge-log}
		\frac1e\sum_{ij\in E(J)}
		\log\bigl(H(d(j),d(i))q_{ij}^{\,r}\bigr)
		\ge
		\log\bigl(H(m,n)\rho^{r+1}\bigr)+r\delta.
	\end{equation}
\end{lemma}

\begin{proof}
	In the logarithmic sums below, terms with coefficient zero
	are interpreted as zero.
		Put
	\[
	h_L=\frac1e\sum_{i\in L}d(i)\log d(i),
	\qquad
	h_R=\frac1e\sum_{j\in R}d(j)\log d(j).
	\]
	Since the degrees in each class sum to $e$, expanding
	the definitions of $S_L,S_R$ gives
	\begin{equation}\label{eq:degree-logs}
		S_L=h_L-\log(e/m),
		\qquad
		S_R=h_R-\log(e/n).
	\end{equation}
	For example, the linear terms in $S_L$ cancel because
	\[
	\sum_{i\in L}\left(-d(i)+\frac em\right)=0.
	\]
		Similarly, the linear terms in $S_C$ cancel, giving
	\begin{equation}\label{eq:weighted-codegree-log}
S_C =\frac1e\sum_{i,i'\in L} W_{ii'}\log\frac{eW_{ii'}}{d(i)d(i')} =\frac1e\sum_{i,i'\in L}W_{ii'}\log W_{ii'} +\log e-2h_L.
\end{equation}
	For the last equality, use
	$\sum_{i'}W_{ii'}=d(i)$ and the symmetry of $W$.
	In particular,
	\[
	\frac1e\sum_{i,i'\in L}W_{ii'}\log d(i)
	=\frac1e\sum_{i\in L}d(i)\log d(i)=h_L,
	\]
	and the term containing $\log d(i')$ has the same value.
		We next compare $C_{ii'}$ and $W_{ii'}$.
	Fix $i,i'$ with $W_{ii'}>0$. The numbers
	\[
	\frac1{d(j)W_{ii'}},
	\qquad j\in N(i)\cap N(i'),
	\]
	are positive and sum to one. Applying the weighted
AM--GM inequality to the degrees
	of these common neighbors gives
	\[
	\log\frac{C_{ii'}}{W_{ii'}}
	\ge
	\frac1{W_{ii'}}
	\sum_{j\in N(i)\cap N(i')}
	\frac{\log d(j)}{d(j)}.
	\]
	Multiplying by $W_{ii'}$, we obtain
	\[
	W_{ii'}\log C_{ii'}
	\ge W_{ii'}\log W_{ii'}
	+\sum_{j\in N(i)\cap N(i')}
	\frac{\log d(j)}{d(j)}.
	\]
	Now sum over $i,i'$ and divide by $e$.
	Each vertex $j\in R$ is a common neighbor of exactly
	$d(j)^2$ ordered pairs in $L$. Therefore
	\begin{equation}\label{eq:common-log}
		\frac1e\sum_{i,i'\in L}W_{ii'}\log C_{ii'}
		\ge
		\frac1e\sum_{i,i'\in L}W_{ii'}\log W_{ii'}+h_R.
	\end{equation}
		For each edge $ij$, count the edges in its full link
	by their endpoint $i'\in N(j)$:
	\[
	e_J(N(j),N(i))=\sum_{i'\in N(j)}C_{ii'}.
	\]
	Hence
	\[
	q_{ij}
	=\frac1{d(i)}
	\left(\frac1{d(j)}
	\sum_{i'\in N(j)}C_{ii'}\right).
	\]
	Every term $C_{ii'}$ in this sum is positive, since $j$
	is a common neighbor of $i$ and $i'$.
	The AM--GM inequality gives
	\[
	\log q_{ij}
	\ge
	\frac1{d(j)}\sum_{i'\in N(j)}\log C_{ii'}
	-\log d(i).
	\]
		Sum over all edges and divide by $e$.
	For fixed $i,i'$, the coefficient of $\log C_{ii'}$ is
	\[
	\frac1e\sum_{j\in N(i)\cap N(i')}\frac1{d(j)}
	=\frac{W_{ii'}}e.
	\]
	Also,
	\[
	\frac1e\sum_{ij\in E(J)}\log d(i)=h_L.
	\]
	It follows from \eqref{eq:common-log},
	\eqref{eq:weighted-codegree-log}, and
	\eqref{eq:degree-logs} that
	\[
\begin{aligned}
 \frac1e\sum_{ij\in E(J)}\log q_{ij}
 &\ge\frac1e\sum_{i,i'\in L}W_{ii'}\log C_{ii'}-h_L\\
 &\ge\frac1e\sum_{i,i'\in L}W_{ii'}\log W_{ii'}+h_R-h_L
 =S_C-\log e+h_L+h_R=\log\rho+S_L+S_R+S_C.
\end{aligned}
\]
	Subtracting $\log\rho$ proves \eqref{eq:log-gap}.
	
	It remains to account for the class sizes of the full links.
	Let $m_+$ and $n_+$ be the numbers of nonisolated vertices
	in $L$ and $R$. Since
	\[
	\frac1{H(d(j),d(i))}
	=\frac12\left(\frac1{d(j)}+\frac1{d(i)}\right),
	\]
	we have
	\[
	\frac1e\sum_{ij\in E(J)}\frac1{H(d(j),d(i))}
	=\frac{m_++n_+}{2e}
	\le\frac{m+n}{2e}
	=\frac1{\rho H(m,n)}.
	\]
		Indeed, each nonisolated vertex $x$ occurs in $d(x)$ edges
	and contributes $1/d(x)$ for each of them.
		Apply concavity of the logarithm to these reciprocals:
		\[
		\frac1e\sum_{ij\in E(J)}\log H(d(j),d(i))
		\ge -\log\left(\frac1e\sum_{ij\in E(J)}
		\frac1{H(d(j),d(i))}\right)
		\ge \log\bigl(\rho H(m,n)\bigr).
		\]
	Finally, \eqref{eq:delta} gives
	\[
	\frac re\sum_{ij\in E(J)}\log q_{ij}
	=r(\log\rho+\delta).
	\]
	Adding this  to the preceding inequality proves
	\eqref{eq:edge-log}.
\end{proof}

\section{Deletion and the preliminary density criterion: Proof of (\ref{eq:ideas-deletion})}\label{sec:deletion}

We prove the deletion estimate \eqref{eq:ideas-deletion}, restated
as \Cref{lem:deletion}, for use in \Cref{thm:one-step}.
We also prove the preliminary density criterion, \Cref{lem:coarse},
which handles the cases where the numerical hypotheses of the
one-step theorem fail.

Fix an edge $xy$, and abbreviate $u=|N(y)|$, $v=|N(x)|$, $q=q_{xy}$.
Let $J'_{xy}$ be the bipartite graph between
$N(y)\setminus\{x\}$ and $N(x)\setminus\{y\}$,
and let $q'$ be its density. If either of these two sets is empty, set
$q'=0$ and interpret every residual potential as zero.
For $k\ge2$, every $K_{k-1,k-1}$ in $J'_{xy}$ extends to a
$K_{k,k}$ in $J$ by adding $x,y$.

\begin{lemma}\label{lem:deletion}
For every real $r\ge1$,
\begin{equation}\label{eq:deletion}
 H(u,v)q^r
 \le H(u-1,v-1)(q')^r+
        \min\{2,\,5r q^{r-1}\}.
\end{equation}
\end{lemma}

\begin{proof}
In the bipartite graph between $N(y)$ and $N(x)$, the vertex $x$ is
adjacent to all $v$ vertices on the other side, and $y$ is adjacent to
all $u$ vertices on its other side. These two stars have exactly
$w=u+v-1$ edges in their union. If $u,v\ge2$, then
\begin{equation}\label{eq:pivotidentity}
 q=\frac w{uv}+\frac{(u-1)(v-1)}{uv}q'.
\end{equation}
In particular, $q'\le q$. Convexity of $z\mapsto z^r$ gives
\[
H(u,v)q^r \le\frac{2w}{u+v} +\frac{2(u-1)(v-1)}{u+v}(q')^r \le2+H(u-1,v-1)(q')^r.
\]
For the refined bound, write $H_0=H(u,v)$ and
$H_1=H(u-1,v-1)$. Direct algebra gives
\[
 H_0-H_1
 =\frac{2(u^2+v^2-u-v)}{(u+v)(u+v-2)}\le2.
\]
Also, \eqref{eq:pivotidentity} implies
\[
 H_1(q-q')
 =\frac{2(1-q)(u+v-1)}{u+v-2}\le3(1-q),
\]
where the last inequality uses $u+v\ge4$.
Since $q^r-(q')^r\le r q^{r-1}(q-q')$, we have
\[
H_0q^r-H_1(q')^r =(H_0-H_1)q^r+H_1\bigl(q^r-(q')^r\bigr) \le2q^r+3r(1-q)q^{r-1} \le5r q^{r-1}.
\]
Combining the two bounds proves \eqref{eq:deletion} when $u,v\ge2$.

If $u=1$ or $v=1$, the full link consists of one of the two stars, so
$q=1$, $H(u,v)<2$, and the residual potential is zero. Since
$\min\{2,5r\}=2$, the same conclusion holds.
\end{proof}

We now give the formal proof of the preliminary density criterion
from the warm-up argument.

\begin{lemma}[Preliminary density criterion]\label{lem:coarse}
For every positive integer $k$ and every bipartite graph $J$,
\begin{equation}\label{eq:coarse}
 B_k(J)>2(k-1)
 \quad\Longrightarrow\quad J\text{ contains }K_{k,k}.
\end{equation}
\end{lemma}

\begin{proof}
We induct on $k$. If $k=1$, then $B_1(J)>0$ implies that $J$
has an edge. Now let $k\ge2$ and assume $B_k(J)>2(k-1)$.
The graph has nonempty classes and positive density. For a
uniformly chosen edge, let $u,v,q$ be the full-link sizes and
density. Taking $r=k-1$ in \eqref{eq:edge-log}, using
$\delta\ge0$ from \Cref{lem:log}, and applying concavity of
the logarithm, we obtain
\[
 \E{\mathrm{edge}}H(u,v)q^{k-1}\ge B_k(J).
\]
By \Cref{lem:deletion},
\[
 \E{\mathrm{edge}}B_{k-1}(J'_{xy})
 \ge B_k(J)-2>2(k-2).
\]
Some edge therefore has $B_{k-1}(J'_{xy})>2(k-2)$.
The induction hypothesis gives a $K_{k-1,k-1}$ in that residual
graph, and adding the selected endpoints gives a $K_{k,k}$ in $J$.
\end{proof}

\section{Controlling the link densities: Proof of (\ref{eq:ideas-stability})}
\label{sec:stability}

The purpose of this section is to prove \eqref{eq:ideas-stability},
restated below as \Cref{lem:stability}. These estimates control the
density errors in the proof of \Cref{thm:one-step}.
Let $J$ have vertex classes $L,R$ of sizes $m,n$, with
$e=|E(J)|>0$ and $\rho=e/(mn)$. Write $d(x)=|N_J(x)|$.
We retain the full-link densities $q_{ij}$, the logarithmic gain
$\delta$, and the quantities $S_L,S_R,S_C$ from \Cref{sec:log}.
In sums without specified ranges, $i,i'$ run over $L$ and $j,j'$
over $R$. Indices from the same class need not be distinct.

\begin{lemma}\label{lem:stability}
	For every bipartite graph with $\rho>0$,
	\begin{equation}\label{eq:stability}
0\le\frac1e\sum_{ij\in E(J)}(q_{ij}-\rho) \le10^4\rho^{-5}\delta, \qquad \frac1e\sum_{ij\in E(J)}(q_{ij}-\rho)^2 \le10^4\rho^{-5}\delta.
\end{equation}
\end{lemma}

\subsection{Degrees, common neighbors, and four-cycles}

The average degrees on $L$ and $R$ are $e/m$ and $e/n$.
Define their sums of squared deviations by
\begin{equation}\label{eq:centered}
	D_L=\sum_i\left(d(i)-\frac em\right)^2,
	\qquad
	D_R=\sum_j\left(d(j)-\frac en\right)^2.
\end{equation}

\begin{lemma}\label{lem:degree-errors}
	We have
	\begin{equation}\label{eq:degree-errors}
		D_L\le2neS_L,\qquad D_R\le2meS_R,\qquad
		mD_L+nD_R\le2mne\delta.
	\end{equation}
\end{lemma}

\begin{proof}
	Since $d(i)\le n$ and $0<e/m\le n$, apply
	\eqref{eq:psi-quadratic} with $s=d(i)$, $t=e/m$, and $M=n$.
	Summing the resulting inequalities gives
	\[
	D_L\le2n\sum_i\psi\left(d(i),\frac em\right)=2neS_L.
	\]
	The right-side estimate follows with $M=m$.
	Multiplying these bounds by $m,n$, respectively, and using
	$S_L+S_R\le\delta$ proves the last inequality.
\end{proof}

We next control the common-neighbor counts
$C_{ii'}=|N(i)\cap N(i')|$ from
\eqref{eq:common-neighbour-sums}. For each ordered pair $i,i'$,
choosing an ordered pair of common neighbors $j,j'$ gives the
closed walk $i,j,i',j',i$. The same walk can be specified by an
edge $ij$ and an edge $i'j'$ of its full link. Counting these
ordered choices in the two ways gives
\begin{equation}\label{eq:four-cycle-count}
	\sum_{i,i'}C_{ii'}^2
	=\sum_{ij\in E(J)}e_J(N(j),N(i)).
\end{equation}
Thus four-cycle counts arise naturally when we sum the edge
counts of full links. Here and below these are labeled
four-cycle counts allowing repeated vertices. 

Regularity alone does not force this count to be close to
$\rho^4m^2n^2$. For example, two
disjoint copies of $K_{N/2,N/2}$ form a bipartite graph with
$N$ vertices in each class, density $1/2$, and every degree
equal to $N/2$. Nevertheless, the sum in
\eqref{eq:four-cycle-count} is $N^4/8$, twice the benchmark
$\rho^4N^4=N^4/16$. We therefore need the common-neighbor
information in $S_C$ as well as the degree information in
$S_L,S_R$.

Let $A_{ij}$ be the adjacency indicator. Define
\begin{equation}\label{eq:B-R}
B_{ii'}=C_{ii'}-\rho d(i)-\rho d(i')+\rho^2n =\sum_j(A_{ij}-\rho)(A_{i'j}-\rho), \qquad R_4=\sum_{i,i'}B_{ii'}^2.
\end{equation}
The subtraction in $B_{ii'}$ centers the adjacency entries;
expanding the square shows that
\[
R_4=\sum_{i,i',j,j'}
(A_{ij}-\rho)(A_{i'j}-\rho)(A_{i'j'}-\rho)(A_{ij'}-\rho).
\]
If every left degree is $\rho n$ and every right degree is
$\rho m$, then $B_{ii'}=C_{ii'}-\rho^2n$ and
$\sum_{i,i'}C_{ii'}=\sum_jd(j)^2=\rho^2m^2n$. Consequently,
\[
R_4=\sum_{i,i'}C_{ii'}^2-\rho^4m^2n^2.
\]
In this special case $R_4$ is exactly the excess four-cycle
count. The following bound also accounts for unequal degrees.

\begin{lemma}\label{lem:fourth}
	We have
	\begin{equation}\label{eq:Rbound}
		R_4\le16e^2S_C+20nD_R+\frac{2D_L^2}{n^2}
		\le60m^2n^2\delta.
	\end{equation}
\end{lemma}

\begin{proof}
	Recall the weighted common-neighbor sums
	\[
	W_{ii'}=\sum_{j\in N(i)\cap N(i')}\frac1{d(j)}.
	\]
	Put
	\[
	\widehat C_{ii'}=\frac en W_{ii'},\qquad
	K_{ii'}=C_{ii'}-\widehat C_{ii'}.
	\]
	The factor $e/n$ is the average right degree. Thus
	$\widehat C_{ii'}=C_{ii'}$ when all right degrees are equal.
	In general,
	\[
	K_{ii'}=\sum_{j\in N(i)\cap N(i')}
	\left(1-\frac{e}{nd(j)}\right).
	\]
	Let $R_+=\{j\in R:d(j)>0\}$. Squaring and changing the order
	of summation gives
	\[
	\sum_{i,i'}K_{ii'}^2
	=\sum_{j,j'\in R_+}
	\left(1-\frac{e}{nd(j)}\right)
	\left(1-\frac{e}{nd(j')}\right)
	|N(j)\cap N(j')|^2.
	\]
	Since $|N(j)\cap N(j')|^2\le d(j)d(j')$, taking absolute
	values of the coefficients and applying Cauchy--Schwarz yields
	\begin{equation}\label{eq:K-square}
	\sum_{i,i'}K_{ii'}^2
			 \le\left(\sum_{j\in R_+}\left|d(j)-\frac en\right|\right)^2	\le n\sum_{j\in R_+}\left(d(j)-\frac en\right)^2
			\le nD_R.
	\end{equation}
		By the identity $\psi(cs,ct)=c\psi(s,t)$ for $c>0$ and
	the definition of $S_C$,
	\[
	\sum_{i,i'}\psi\left(\widehat C_{ii'},\frac{d(i)d(i')}{n}\right)
	=\frac en\sum_{i,i'}
	\psi\left(W_{ii'},\frac{d(i)d(i')}{e}\right)
	=\frac{e^2}{n}S_C.
	\]
	For pairs with $\widehat C_{ii'}\le2n$, both arguments of
	$\psi$ are at most $2n$. Hence \eqref{eq:psi-quadratic} gives
	\[
	\left(\widehat C_{ii'}-\frac{d(i)d(i')}{n}\right)^2
	\le4n\,\psi\left(\widehat C_{ii'},\frac{d(i)d(i')}{n}\right).
	\]
	If $d(i)d(i')=0$, both sides are zero.
	For pairs with $\widehat C_{ii'}>2n$, the bounds
	$C_{ii'}\le n$ and $d(i)d(i')/n\le n$ give
	\[
	|K_{ii'}|\ge\frac{\widehat C_{ii'}}2,
	\qquad
	\left(\widehat C_{ii'}-\frac{d(i)d(i')}{n}\right)^2
	\le\widehat C_{ii'}^{\,2}\le4K_{ii'}^2.
	\]
	Summing over these two sets of pairs and using
	\eqref{eq:K-square}, we obtain
	\begin{equation}\label{eq:weighted-common-square}
		\sum_{i,i'}
		\left(\widehat C_{ii'}-\frac{d(i)d(i')}{n}\right)^2
		\le4e^2S_C+4nD_R.
	\end{equation}
	As $C_{ii'}=K_{ii'}+\widehat C_{ii'}$, the inequality
	$(x+y)^2\le2x^2+2y^2$ now gives
	\begin{equation}\label{eq:common-square}
		\sum_{i,i'}\left(C_{ii'}-\frac{d(i)d(i')}{n}\right)^2
		\le8e^2S_C+10nD_R.
	\end{equation}
	Finally,
	\[
	B_{ii'}=C_{ii'}-\frac{d(i)d(i')}{n}
	+\frac1n\left(d(i)-\frac em\right)
	\left(d(i')-\frac em\right).
	\]
	Applying the same inequality once more gives
	\[
	R_4\le16e^2S_C+20nD_R+\frac{2D_L^2}{n^2}.
	\]
	Since $e\le mn$, $D_L\le mn^2$, and $S_L,S_R,S_C\le\delta$,
	\Cref{lem:degree-errors} bounds the three terms by
	$16m^2n^2\delta$, $40m^2n^2\delta$, and $4m^2n^2\delta$,
	respectively. This proves \eqref{eq:Rbound}.
\end{proof}

\begin{remark}\label{sec:boxnorm}
	The ratio $R_4/(m^2n^2)$ is the fourth power of the rectangular
	box norm of $A-\rho$, the centered four-cycle expression related
	to quasirandomness~\cite{CGW,Gowers}. Here \eqref{eq:Rbound}
	supplies the linear dependence on $\delta$ needed below.
\end{remark}

\subsection{Controlling full-link edge counts}

The full link of an edge $ij$ has $d(i)d(j)$ possible edges.
To compare its density with $\rho$, we therefore compare its
edge count with $\rho d(i)d(j)$. In this section, we use the  degree
and common-neighbor estimates from the previous section to bound both the sum of
squared edge-count errors and their signed sum over edges.
Both bounds are linear in $\delta$.

These estimates are the next step towards controlling
$q_{ij}-\rho$: the following subsection divides the edge-count
errors by $d(i)d(j)$ and controls the effect of these varying
denominators.

\begin{lemma}\label{lem:numerator}
	For every $i\in L$ and $j\in R$, define
	\begin{equation}\label{eq:T}
		T_{ij}=e_J(N(j),N(i))-\rho d(i)d(j).
	\end{equation}
	Then
	\begin{equation}\label{eq:T-bounds}
		\sum_{i,j}T_{ij}^2\le186m^3n^3\delta,
		\qquad
		\sum_{ij\in E(J)}T_{ij}\le65m^2n^2\delta.
	\end{equation}
	On every edge $ij$,
	\begin{equation}\label{eq:q-T}
		q_{ij}-\rho=\frac{T_{ij}}{d(i)d(j)},
	\end{equation}
	and for every pair $i,j$,
	\begin{equation}\label{eq:T-pointwise}
		|T_{ij}|\le d(i)d(j)\le mn.
	\end{equation}
\end{lemma}

\begin{proof}
	The definition of the full-link density gives \eqref{eq:q-T}.
	Also, $0\le e_J(N(j),N(i))\le d(i)d(j)$, so both terms in
	\eqref{eq:T} lie between zero and $d(i)d(j)$. This proves
	\eqref{eq:T-pointwise}.
	
	To prove the two sum bounds, put
	\[
	\xi_i=\sum_{j\in N(i)}\left(d(j)-\frac en\right),\qquad
	\eta_j=\sum_{i\in N(j)}\left(d(i)-\frac em\right),\qquad
	D_{ij}=\sum_{i'}B_{ii'}(A_{i'j}-\rho).
	\]
	The first two quantities sum degree deviations over a
	neighborhood. Counting common neighbors first gives
	\[
	\sum_{i'}C_{ii'}=\sum_{j\in N(i)}d(j),
	\qquad
	\sum_{i'}B_{ii'}=\xi_i.
	\]
	Moreover,
	$e_J(N(j),N(i))=\sum_{i'\in N(j)}C_{ii'}$.
	Substituting the definition of $B_{ii'}$ gives
	\begin{equation}\label{eq:T-expansion}
		T_{ij}=\sum_{i'\in N(j)}B_{ii'}+\rho\eta_j
		=\rho\xi_i+\rho\eta_j+D_{ij}.
	\end{equation}
		By Cauchy--Schwarz on each neighborhood,
	\[
		\sum_i\xi_i^2
		\le\sum_i d(i)\sum_{j\in N(i)}\left(d(j)-\frac en\right)^2
		\le n\sum_j d(j)\left(d(j)-\frac en\right)^2
		\le mnD_R.
	\]
	Similarly, $\sum_j\eta_j^2\le mnD_L$.
	For $D_{ij}$, Cauchy--Schwarz gives
	\[
	D_{ij}^2\le
	\left(\sum_{i'}B_{ii'}^2\right)
	\left(\sum_{i'}(A_{i'j}-\rho)^2\right).
	\]
	The factors depend only on $i$ and $j$, respectively. Since
	\begin{equation}\label{eq:U-square}
		\sum_{i,j}(A_{ij}-\rho)^2=e(1-\rho)\le mn,
	\end{equation}
	summing the preceding inequality yields
	\[
	\sum_{i,j}D_{ij}^2
	\le\left(\sum_{i,i'}B_{ii'}^2\right)
	\left(\sum_{i',j}(A_{i'j}-\rho)^2\right)
	\le mnR_4.
	\]
	Consequently, \eqref{eq:T-expansion} and
	$(x+y+z)^2\le3(x^2+y^2+z^2)$ imply
	\[
\begin{aligned}
 \sum_{i,j}T_{ij}^2
 &\le3\rho^2\left(n\sum_i\xi_i^2+m\sum_j\eta_j^2\right)+3\sum_{i,j}D_{ij}^2
 \le3mn\bigl(\rho^2(mD_L+nD_R)+R_4\bigr)\\
 &\le6\rho^2m^2n^2e\delta+180m^3n^3\delta\le186m^3n^3\delta.
\end{aligned}
\]
	Here we used \Cref{lem:degree-errors,lem:fourth} and $e\le mn$.
		For the signed sum, define the mixed degree error
	\begin{equation}\label{eq:gamma}
	\Gamma
			=\sum_{ij\in E(J)}
			\left(d(i)-\frac em\right)\left(d(j)-\frac en\right)
			=\sum_i\left(d(i)-\frac em\right)\xi_i
			=\sum_j\left(d(j)-\frac en\right)\eta_j.
		\end{equation}
	Changing the order of summation gives
	\[
	\sum_i\xi_i=\sum_jd(j)\left(d(j)-\frac en\right)=D_R,
	\qquad
	\sum_j\eta_j=D_L.
	\]
	It follows that
	\[
	\sum_{ij\in E(J)}\xi_i=\frac emD_R+\Gamma,
	\qquad
	\sum_{ij\in E(J)}\eta_j=\frac enD_L+\Gamma.
	\]
	Also, by the definitions of $D_{ij}$ and $B_{ii'}$,
	\[
		\sum_{ij\in E(J)}D_{ij}
		=\sum_{i,i'}B_{ii'}\bigl(C_{ii'}-\rho d(i)\bigr)
		=R_4+\rho\sum_{i'}\left(d(i')-\frac em\right)
		\sum_iB_{ii'}
		=R_4+\rho\Gamma.
	\]
	For the last equality, use the symmetry of $B$ and
	$\sum_iB_{ii'}=\xi_{i'}$. Summing \eqref{eq:T-expansion}
	over the edges therefore gives the exact identity
	\begin{equation}\label{eq:AT}
		\sum_{ij\in E(J)}T_{ij}
		=\rho^2(mD_L+nD_R)+3\rho\Gamma+R_4.
	\end{equation}
	Finally, Cauchy--Schwarz and the bound on $\sum_i\xi_i^2$ give
	\[
	|\Gamma|\le\sqrt{D_L\sum_i\xi_i^2}
	\le\sqrt{mnD_LD_R}
	\le\frac{mD_L+nD_R}{2}\le mne\delta.
	\]
	Substitution into \eqref{eq:AT} yields
	\[
	\sum_{ij\in E(J)}T_{ij}
	\le2\rho^2mne\delta+3\rho mne\delta+60m^2n^2\delta
	\le65m^2n^2\delta,
	\]
	as required.
\end{proof}

\subsection{Completing the proof of Lemma~\ref{lem:stability}}

The preceding subsection controls the edge-count errors
$T_{ij}$. To obtain the required density estimates, we use
\eqref{eq:q-T}:
\[
q_{ij}-\rho=\frac{T_{ij}}{d(i)d(j)}.
\]
The remaining difficulty is that the denominator varies with
the edge and may be small. We first use the degree estimates
to control the contribution of edges with a small endpoint
degree. On the remaining edges, a lower bound on $d(i)d(j)$
suffices to bound the sum of squared density errors.

The signed sum requires more care. We compare $1/(d(i)d(j))$
with the constant $mn/e^2$, the reciprocal of the product of
the average degrees. This allows us to use the signed-sum
bound from \Cref{lem:numerator} and then control the error
introduced by replacing the denominator.

	\begin{proof}[Proof of \Cref{lem:stability}]
		Recall that we must show
		\[
		0\le\frac1e\sum_{ij\in E(J)}(q_{ij}-\rho)\le10^4\rho^{-5}\delta,
		\qquad
		\frac1e\sum_{ij\in E(J)}(q_{ij}-\rho)^2\le10^4\rho^{-5}\delta.
		\]
		We first bound the sum of squares, then prove the upper
		and lower bounds for the signed sum.
		Let $\mathcal B$ be the set of edges $ij$ with
	$d(i)<e/(2m)$ or $d(j)<e/(2n)$.
	There are at most $4m^2D_L/e^2$ left vertices with
	$d(i)<e/(2m)$, since each contributes more than
	$e^2/(4m^2)$ to $D_L$. Each has fewer than $e/(2m)$ incident
	edges, so together they meet at most $2mD_L/e$ edges.
	The corresponding right-side count is at most $2nD_R/e$.
	Thus
	\begin{equation}\label{eq:bad-fraction}
		|\mathcal B|\le\frac{2(mD_L+nD_R)}e\le4mn\delta,
		\qquad
		\frac{|\mathcal B|}{e}\le4\rho^{-1}\delta.
	\end{equation}
		Outside $\mathcal B$, we have $d(i)d(j)\ge e^2/(4mn)$.
	Since $|q_{ij}-\rho|\le1$ on every edge,
	\eqref{eq:q-T} and \eqref{eq:T-bounds} give
	\begin{equation}\label{eq:second-final}
\begin{aligned}
 \frac1e\sum_{ij\in E(J)}(q_{ij}-\rho)^2
 &\le\frac{|\mathcal B|}{e}
 +\frac1e\sum_{ij\in E(J)\setminus\mathcal B}\frac{T_{ij}^2}{d(i)^2d(j)^2}\\
 &\le4\rho^{-1}\delta+\frac{16m^2n^2}{e^5}\sum_{i,j}T_{ij}^2
 \le4\rho^{-1}\delta+2976\rho^{-5}\delta\le2980\rho^{-5}\delta.
\end{aligned}
\end{equation}
		For the signed sum, applying Cauchy--Schwarz directly to
	\eqref{eq:second-final} would give a bound proportional to
	$\sqrt\delta$. 
		We split the signed average into contributions from
	$\mathcal B$ and its complement
	$\mathcal G=E(J)\setminus\mathcal B$. Define
	\[
	M_{\mathcal B}=\frac1e\sum_{ij\in\mathcal B}(q_{ij}-\rho),
	\qquad
	M_{\mathcal G}=\frac1e\sum_{ij\in\mathcal G}(q_{ij}-\rho).
	\]
	Thus
	\[
	\frac1e\sum_{ij\in E(J)}(q_{ij}-\rho)
	=M_{\mathcal B}+M_{\mathcal G}.
	\]
	We bound these two contributions separately.
		First, on $\mathcal B$, the bound $|q_{ij}-\rho|\le1$
	and \eqref{eq:bad-fraction} give
	\[
	|M_{\mathcal B}|
	\le\frac{|\mathcal B|}{e}
	\le4\rho^{-1}\delta.
	\]
		Next, consider $M_{\mathcal G}$. By \eqref{eq:q-T},
	\[
	M_{\mathcal G}
	=\frac1e\sum_{ij\in\mathcal G}\frac{T_{ij}}{d(i)d(j)}.
	\]
	We compare $d(i)d(j)$ with the product of the average
	degrees, $e^2/(mn)$, and write
	\[
	M_{\mathcal G}
	=\frac{mn}{e^3}\sum_{ij\in\mathcal G}T_{ij}+R_{\mathcal G},
	\]
	where
	\[
	R_{\mathcal G}
	=\frac1e\sum_{ij\in\mathcal G}
	T_{ij}\left(\frac1{d(i)d(j)}-\frac{mn}{e^2}\right).
	\]	We now bound these two terms.
		The signed-sum estimate in \eqref{eq:T-bounds} applies
	to all edges. To use it for $\mathcal G$, observe that
	\[
	\sum_{ij\in\mathcal G}T_{ij}
	=\sum_{ij\in E(J)}T_{ij}-\sum_{ij\in\mathcal B}T_{ij}.
	\]
	Since $|T_{ij}|\le mn$ and $|\mathcal B|\le4mn\delta$,
	\[
	\frac{mn}{e^3}\sum_{ij\in\mathcal G}T_{ij}
	\le\frac{mn}{e^3}
	\left(65m^2n^2\delta+mn|\mathcal B|\right)
	\le69\rho^{-3}\delta.
	\]
		It remains to bound $R_{\mathcal G}$. For $ij\in\mathcal G$,
	we have $d(i)d(j)\ge e^2/(4mn)$, so
	\[
	\left|\frac1{d(i)d(j)}-\frac{mn}{e^2}\right|
	\le\frac{4m^2n^2}{e^4}
	\left|d(i)d(j)-\frac{e^2}{mn}\right|.
	\]
	Using
	\[
	d(i)d(j)-\frac{e^2}{mn}
	=d(i)\left(d(j)-\frac en\right)
	+\frac en\left(d(i)-\frac em\right),
	\]
	together with $d(i)\le n$ and $e/n\le m$, we obtain
	\[
	\left|d(i)d(j)-\frac{e^2}{mn}\right|
	\le m\left|d(i)-\frac em\right|
	+n\left|d(j)-\frac en\right|.
	\]
		By \Cref{lem:degree-errors} and $e\le mn$,
	\[
	\sum_{i,j}
	\left(m\left|d(i)-\frac em\right|
	+n\left|d(j)-\frac en\right|\right)^2
	\le2mn(mD_L+nD_R)
	\le4m^3n^3\delta.
	\]
	Here the sum is over all pairs in $L\times R$.
	Applying the reciprocal bound and extending the resulting
	nonnegative sum from $\mathcal G$ to all such pairs gives
	\[
	|R_{\mathcal G}|
	\le\frac{4m^2n^2}{e^5}\sum_{i,j}|T_{ij}|
	\left(m\left|d(i)-\frac em\right|
	+n\left|d(j)-\frac en\right|\right).
	\]
	Cauchy--Schwarz and \eqref{eq:T-bounds} now yield
	\[
	|R_{\mathcal G}|
	\le\frac{4m^2n^2}{e^5}
	\sqrt{186m^3n^3\delta\cdot4m^3n^3\delta}
	\le112\rho^{-5}\delta.
	\]
	Consequently,
	\[
	M_{\mathcal G}\le69\rho^{-3}\delta+112\rho^{-5}\delta.
	\]
	Combining the bounds for $M_{\mathcal B}$ and
	$M_{\mathcal G}$ gives
	\begin{equation}\label{eq:meanfinal}
		\frac1e\sum_{ij\in E(J)}(q_{ij}-\rho)
		\le4\rho^{-1}\delta+69\rho^{-3}\delta+112\rho^{-5}\delta
		\le185\rho^{-5}\delta.
	\end{equation}
	Finally, \Cref{lem:AMGM} and \eqref{eq:log-gap} give
	\[
	\frac1e\sum_{ij\in E(J)}q_{ij}
	\ge\exp\!\left(\frac1e\sum_{ij\in E(J)}\log q_{ij}\right)
	=\rho\exp(\delta)\ge\rho.
	\]
	This proves the lower bound in \eqref{eq:stability}; its two
	upper bounds follow from \eqref{eq:second-final} and
	\eqref{eq:meanfinal}.
\end{proof}

\section{Proof of the one-step density theorem}\label{sec:reduction}

We now combine the estimates from
Sections~\ref{sec:log}--\ref{sec:stability} to prove
\Cref{thm:one-step}, with the potential and thresholds defined
in \eqref{eq:potential} and \eqref{eq:Fdef}.
After proving the theorem, we deduce a density criterion for
arbitrary bipartite graphs and then the main results.

\subsection{Verifying the one-step inequality}\label{sec:induction}

\begin{proof}[Proof of \Cref{thm:one-step}]
	Let $m,n,\rho$ be the class sizes and density of $G$, and write
	$e=e(G)$. Since $\calP_k(G)>F_k$, we have $e>0$ and $0<\rho<1$.
	For each edge $xy$, put $u=d_G(y)$ and $v=d_G(x)$, and let
	$q,q'$ be the densities of its full and residual links,
	respectively. Set $q'=0$ if a residual class is empty.
	These quantities depend on $xy$ in every edge sum below.
	
	The additive loss we allow is
	\begin{equation}\label{eq:expected-notation}
		\varepsilon=\frac{F_{k-1}}{k^2},
	\end{equation}
	and the corrected residual potential is
	\begin{equation}\label{eq:edge-quantities}
		S_{xy}=\calP_{k-1}(G'_{xy})
		=H(u-1,v-1)(q')^{k+19}(1-q')^{20}.
	\end{equation}
		In particular, $S_{xy}=0$ if a residual class is empty.
	Our goal is
	\[
	\frac1e\sum_{xy\in E(G)}S_{xy}
	\ge\calP_k(G)-\varepsilon>F_{k-1}.
	\]
	It suffices to prove the stronger logarithmic inequality
	\begin{equation}\label{eq:one-step-log-target}
		\frac1e\sum_{xy\in E(G)}\log(S_{xy}+\varepsilon)
		\ge\log\calP_k(G).
	\end{equation}
	The positive shift $\varepsilon$ makes every logarithm
	defined, even when $S_{xy}=0$.
		Indeed, once \eqref{eq:one-step-log-target} is established,
	concavity of the logarithm, applied with equal weights $1/e$
	to the $e$ edges, gives
	\[
	\log\calP_k(G)
	\le\frac1e\sum_{xy\in E(G)}\log(S_{xy}+\varepsilon)
	\le\log\left(\frac1e\sum_{xy\in E(G)}S_{xy}+\varepsilon\right).
	\]
	Here we used that $\varepsilon$ is the same for every edge.
	Exponentiating and subtracting $\varepsilon$ gives the
	required lower bound on the arithmetic average. The strict
	inequality follows from $\calP_k(G)>F_k$ and
	$F_k-F_{k-1}>\varepsilon$.
		The inequalities will be proved differently in two different ranges of full-link densities:
	$q\le(1+\rho)/2$, and
	$q>(1+\rho)/2$. 
	
	\medskip
	\noindent\emph{1. The available estimates.}
	We use three previously established estimates. First, apply
	\eqref{eq:edge-log} with exponent $k+19=(k-1)+20$ and add
	$20\log(1-\rho)$ to obtain
	\begin{equation}\label{eq:Zlower}
		\frac1e\sum_{xy\in E(G)}
		\log\bigl(H(u,v)q^{k+19}(1-\rho)^{20}\bigr)
		\ge\log\calP_k(G)+(k+19)\delta.
	\end{equation}
		Second, \Cref{lem:deletion} gives
	\[
	H(u,v)q^{k+19}
	\le H(u-1,v-1)(q')^{k+19}
	+\min\{2,5(k+19)q^{k+18}\}.
	\]
	Third, \Cref{lem:stability} gives
	\[
	\begin{aligned}
		0\le\frac1e\sum_{xy\in E(G)}(q-\rho)
		&\le10^4\rho^{-5}\delta,\\
		\frac1e\sum_{xy\in E(G)}(q-\rho)^2
		&\le10^4\rho^{-5}\delta.
	\end{aligned}
	\]
	We will use the last two estimates to show that passing from
	the expression inside the logarithm in \eqref{eq:Zlower} to
	$S_{xy}+\varepsilon$ loses at most $k\delta$ on average.
	
	Recall that $A=2^{40}$ and $F_j\ge A$ for every $j\ge1$.
	Now  \eqref{eq:parent-density-hyp}
	implies
	\begin{equation}\label{eq:nontrivial}
		A\le F_k<\calP_k(G)
		=B_k(G)[\rho(1-\rho)]^{20}
		\le2(k-1)[\rho(1-\rho)]^{20}.
	\end{equation}
	Consequently,
	\begin{equation}\label{eq:pzlower}
		k>A/2,\qquad [\rho(1-\rho)]^{20}>\frac{A}{2k}.
	\end{equation}
	Thus, under the bound $B_k(G)\le2(k-1)$, the condition
	$\calP_k(G)>F_k$ forces the parent density $\rho$

	away from both
	endpoints, by an amount depending on $k$.
	This restriction will control both the refined deletion loss
	and the coefficient of the averaged density error.
	
	\medskip
	\noindent\emph{2. Links with $q\le(1+\rho)/2$.}
	We first bound the deletion error $5(k+19)q^{k+18}$ by
	$\varepsilon$. From \eqref{eq:pzlower},
	$(1-\rho)^{20}>A/(2k)$. Since $0<1-\rho<1$, we obtain
	\begin{equation}\label{eq:zlower}
		1-\rho\ge\sqrt{\frac{A}{2k}}.
	\end{equation}
	Using $q\le1-(1-\rho)/2$, $1-s\le\exp(-s)$,
	$F_{k-1}\ge A$, and $k>A/2>19$, we get
	\begin{equation}\label{eq:deletion-exp}
		\frac{5(k+19)q^{k+18}}{F_{k-1}}
		\le\frac{10k}{A}\exp\!\left(-\frac{k(1-\rho)}2\right)
		\le\frac{10k}{A}\exp\!\left(-\sqrt{\frac{Ak}{8}}\right).
	\end{equation}
	The inequality $\exp(x)\ge x^6/6!$ bounds the final expression by
	\[
	\frac{10\cdot8^3\cdot6!}{A^4k^2}
	=\frac{3{,}686{,}400}{A^4k^2}\le\frac1{k^2}.
	\]
	Hence
	\begin{equation}\label{eq:small-deletion}
		5(k+19)q^{k+18}\le\varepsilon.
	\end{equation}
	Since $q'\le q$, multiplying the deletion bound by
	$(1-q)^{20}$ gives
	\begin{equation}\label{eq:lowq}
		H(u,v)q^{k+19}(1-q)^{20}
		\le S_{xy}+\varepsilon(1-q)^{20}
		\le S_{xy}+\varepsilon.
	\end{equation}
	Thus the factor $(1-q')^{20}$ causes no additional deletion
	loss. To compare with the parent factor, note that on the
	interval between $\rho$ and $q$, we have $1-s\ge(1-\rho)/2$.
	The second derivative of $20\log(1-s)$ is therefore at least
	$-80/(1-\rho)^2$, so Taylor's inequality gives
	\begin{equation}\label{eq:Taylor}
		20\log(1-q)
		\ge20\log(1-\rho)-\frac{20(q-\rho)}{1-\rho}
		-\frac{40(q-\rho)^2}{(1-\rho)^2}.
	\end{equation}
	Combining this with the logarithm of \eqref{eq:lowq} bounds
	the loss by a linear and a quadratic term in $q-\rho$.
	
	\medskip
	\noindent\emph{3. Links with $q>(1+\rho)/2$.}
	Here we use the constant deletion bound and the residual
	density hypothesis \eqref{eq:residual-density-hyp}:
	\begin{equation}\label{eq:Zcap}
		H(u,v)q^{k+19}
		\le B_{k-1}(G'_{xy})(q')^{20}+2
		\le2(k-2)+2\le2k.
	\end{equation}
	Since $S_{xy}+\varepsilon\ge\varepsilon\ge1/k^2$, we have
	\[
	\log\frac{H(u,v)q^{k+19}(1-\rho)^{20}}
	{S_{xy}+\varepsilon}
	\le\log\frac{2k}{\varepsilon}
	\le\log(2k^3)\le3\log(2k).
	\]
	The inequality $q-\rho>(1-\rho)/2$ now gives
	\[
	3\log(2k)
	\le\frac{12\log(2k)}{(1-\rho)^2}(q-\rho)^2.
	\]
	The logarithmic loss is therefore bounded by a multiple of
	$(q-\rho)^2$, even when $q=1$ or $S_{xy}=0$.
	
	\medskip
	\noindent\emph{4. Averaging and completing the proof.}
	The estimates in the two ranges give, for every edge,
	\begin{equation}\label{eq:globalTaylor}
		\begin{aligned}
			\log(S_{xy}+\varepsilon)
			&\ge\log\bigl(H(u,v)q^{k+19}(1-\rho)^{20}\bigr)\\
			&\quad-\frac{20(q-\rho)}{1-\rho}
			-\frac{40+12\log(2k)}{(1-\rho)^2}(q-\rho)^2.
		\end{aligned}
	\end{equation}
	In the first range, the extra quadratic term only weakens
	the bound from \eqref{eq:lowq} and \eqref{eq:Taylor}. In the
	second range, the additional terms subtracted on the right
	are nonnegative because $q-\rho>0$. Having one inequality
	over all edges preserves the cancellation in the signed
	linear term when we average.
	
	By \Cref{lem:stability}, the average of the two error terms
	in \eqref{eq:globalTaylor} is at most
	\begin{equation}\label{eq:Ldef}
		10^4\rho^{-5}\left(\frac{20}{1-\rho}
		+\frac{40+12\log(2k)}{(1-\rho)^2}\right)\delta
		\le\frac{10^4\rho^{-5}}{(1-\rho)^2}
		\bigl(60+12\log(2k)\bigr)\delta.
	\end{equation}
	To show that this is at most $k\delta$, use the density
	restriction \eqref{eq:pzlower}:
	\begin{equation}\label{eq:Lratio}
		\rho^{-5}(1-\rho)^{-2}
		\le[\rho(1-\rho)]^{-5}
		<\left(\frac{2k}{A}\right)^{1/4}.
	\end{equation}
	The coefficient of $\delta$ on the right of \eqref{eq:Ldef},
	divided by $k$, is therefore at most
	\[
	10^4\left(\frac2A\right)^{1/4}
	k^{-3/4}\bigl(60+12\log(2k)\bigr).
	\]
	The function $x^{-3/4}(60+12\log(2x))$ decreases for $x\ge1$,
	since its derivative has the sign of
	$12-\frac34(60+12\log(2x))<0$. As $k>A/2$, the last display
	is at most
	\begin{equation}\label{eq:Lsmall}
		\frac{2\cdot10^4}{A}(60+12\log A)<1.
	\end{equation}
	Indeed, $\log A<28$ and $A=2^{40}>10^{12}$, whereas the
	numerator is less than $7{,}920{,}000$.
	
	Averaging \eqref{eq:globalTaylor} consequently gives
	\begin{equation}\label{eq:Zupper}
		\frac1e\sum_{xy\in E(G)}\log(S_{xy}+\varepsilon)
		\ge\frac1e\sum_{xy\in E(G)}
		\log\bigl(H(u,v)q^{k+19}(1-\rho)^{20}\bigr)-k\delta.
	\end{equation}
	Combining this with \eqref{eq:Zlower} and $\delta\ge0$ yields
	\begin{equation}\label{eq:shifted-log-potential}
		\frac1e\sum_{xy\in E(G)}\log(S_{xy}+\varepsilon)
		\ge\log\calP_k(G)+19\delta\ge\log\calP_k(G),
	\end{equation}
	which proves \eqref{eq:one-step-log-target} and, by the
	reduction at the beginning, completes the proof.
\end{proof}

\subsection{Deducing a biclique from the potential}\label{sec:application}

We next combine \Cref{thm:one-step} with the preliminary density
criterion to obtain a statement for arbitrary bipartite graphs.

\begin{corollary}\label{thm:density}
For every positive integer $k$ and every bipartite graph $G$,
\begin{equation}\label{eq:mainpotential}
 \calP_k(G)>F_k
 \quad\Longrightarrow\quad G\text{ contains }K_{k,k}.
\end{equation}
\end{corollary}

\begin{proof}
We induct on $k$. For $k=1$, positive potential implies positive
density, so $G$ has an edge. Now let $k\ge2$, assume the result
for $k-1$, and suppose $\calP_k(G)>F_k$.

If $B_k(G)>2(k-1)$, then \Cref{lem:coarse} gives a $K_{k,k}$
in $G$ directly. We may therefore assume
$B_k(G)\le2(k-1)$.

If some edge $xy$ has $B_{k-1}(G'_{xy})>2(k-2)$, then
\Cref{lem:coarse}, applied with $k-1$ to that residual link,
gives a $K_{k-1,k-1}$. Adding $x,y$ yields a $K_{k,k}$ in $G$.
We may therefore also assume
$B_{k-1}(G'_{xy})\le2(k-2)$ for every edge.

The hypotheses of \Cref{thm:one-step} now hold, so some edge
satisfies $\calP_{k-1}(G'_{xy})>F_{k-1}$.
The induction hypothesis gives a $K_{k-1,k-1}$ in that residual
link, which again extends to a $K_{k,k}$ in $G$.
\end{proof}

\subsection{Substitution of \texorpdfstring{$N=c2^t$}{N = c times 2 to the t}}
\label{sec:half-application}

We now deduce \Cref{thm:half,thm:ramsey} from \Cref{thm:density}.
For a graph with two classes of size $N$ and density exactly $1/2$,
the potential is
\[
 \calP_t(G)=N2^{-t-40}.
\]
Thus $N\ge2^{t+81}$ makes the potential at least $2^{41}>F_t$.
The following reduction handles starting density at least $1/2$.

\begin{proof}[Proof of \Cref{thm:half}]
Put $M=2^{t+81}$. If $N\ge M$ and $e(G)\ge N^2/2$, sum $e_G(S,T)$
over all $M$-element sets $S\subseteq L$, $T\subseteq R$. Every
edge is counted $\binom{N-1}{M-1}^2$ times, so the mean is
\[
 \frac{e(G)\binom{N-1}{M-1}^2}{\binom NM^2}
 =e(G)(M/N)^2\ge M^2/2.
\]
Some pair $S,T$ therefore spans at least $M^2/2$ edges. Since $M$
is even, keep exactly $M^2/2$ of those edges. The resulting graph
$J$ has two classes of size $M$ and density exactly $1/2$, so
\[
 \calP_t(J)=M2^{-t-40}=2^{41}>F_t.
\]
By \Cref{thm:density}, $J$ contains a $K_{t,t}$, and hence so does $G$.
\end{proof}

The passage to exactly half the edges is necessary because the
potential in \eqref{eq:potential} is not increasing throughout $0<\rho<1$.
We do not simply substitute $\rho=1/2$ into an inequality whose actual
density might be larger.

\begin{proof}[Proof of \Cref{thm:ramsey}]
In a red--blue coloring of $K_{N,N}$, one color has at least
$N^2/2$ edges. For $N=2^{t+81}$, \Cref{thm:half} applied to that
color gives a monochromatic $K_{t,t}$.
\end{proof}

\appendix
\section{Standard  inequalities}\label{app:scalar}

The purpose of this appendix is to record and prove the two
elementary convexity inequalities used in the main text. All graph-specific lemmas
and their proofs are in the sections where they are needed.

\begin{lemma}\label{lem:AMGM}
If $\lambda_1,\ldots,\lambda_s\ge0$ have sum one and
$x_1,\ldots,x_s>0$, then
\begin{equation}\label{eq:AMGM}
 \sum_{a=1}^s\lambda_a\log x_a
 \le\log\left(\sum_{a=1}^s\lambda_a x_a\right).
\end{equation}
\end{lemma}

\begin{proof}
Put $M=\sum_a\lambda_a x_a$. The elementary inequality
$\log z\le z-1$, applied to $z=x_a/M$ and multiplied by $\lambda_a$,
gives
\[
 \sum_a\lambda_a\log(x_a/M)
 \le\sum_a\lambda_a(x_a/M-1)=0.
\]
This is \eqref{eq:AMGM}.
\end{proof}

\begin{lemma}\label{lem:remainder}
For $s\ge0$ and $t>0$, put
\begin{equation}\label{eq:psi}
 \psi(s,t)=s\log(s/t)-s+t,
\end{equation}
where $0\log(0/t)=0$. Then $\psi(s,t)\ge0$.
If $0\le s\le M$ and $0<t\le M$, then
\begin{equation}\label{eq:psi-quadratic}
 \psi(s,t)\ge\frac{(s-t)^2}{2M}.
\end{equation}
We also put $\psi(0,0)=0$.
\end{lemma}

\begin{proof}
For $F(x)=x\log x$, the expression in \eqref{eq:psi} is
\[
 F(s)-F(t)-F'(t)(s-t).
\]
Since $F''(x)=1/x>0$ for $x>0$, this expression is nonnegative.
On $(0,M]$, we have $F''(x)\ge1/M$, so Taylor's inequality gives
\eqref{eq:psi-quadratic}. The statements at $s=0$ follow by
continuity.
\end{proof}

\end{document}